%% file: LASSO.tex
\documentclass[letterpaper, 10 pt, conference]{ieeeconf}

\IEEEoverridecommandlockouts                              % This command is only
\usepackage{enumerate}
\usepackage{amsmath,amssymb,algorithm,cite,cases,bm,float,subfigure,graphicx,url}
\usepackage{pbox}
\newtheorem{defn}{Definition}[section]

\newtheorem{thm}{Theorem}[section]
\newtheorem{lem}{Lemma}[section]

\usepackage{amsmath,amssymb,algorithm,cite,subfigure,caption,cases,bm,float,subfigure,graphicx,url,color}
\usepackage[hidelinks]{hyperref}
\usepackage{comment}
\usepackage{thm-restate, enumerate }

\definecolor{darkred}{RGB}{150,0,0}
\definecolor{darkgreen}{RGB}{0,150,0}
\definecolor{darkblue}{RGB}{0,0,200}
\hypersetup{colorlinks=true, linkcolor=darkred, citecolor=darkgreen, urlcolor=darkblue}

\input{commands}

\begin{document}
\title{\LARGE{\bf{Simple Error Bounds for Regularized Noisy Linear Inverse Problems}}\vs{1pt}}
%Other title possibilities
%The risk of noisy linear inverse algorithms
%
%\vspace
\author{Christos Thrampoulidis, Samet Oymak and Babak Hassibi\vs{2pt}\\Department of Electrical Engineering, Caltech, Pasadena -- 91125
\thanks{
Emails: {\tt(cthrampo, soymak, hassibi)@caltech.edu}. 
This work was supported in part by the National Science Foundation under grants CCF-0729203, CNS-0932428 and CIF-1018927, by the Office of Naval Research under the MURI grant N00014-08-1-0747, and by a grant from Qualcomm Inc.}
}

%\address{}
%\date{}
%

\maketitle
\begin{abstract} 
      Consider estimating a structured signal $\x_0$ from linear, underdetermined and noisy measurements $\y=\A\x_0+\z$, via solving a variant of the lasso algorithm: $\hat\x=\arg\min_\x\{ \|\y-\A\x\|_2+\la f(\x)\}$. Here, $f$ is a convex function aiming to promote the structure of $\x_0$, say $\ell_1$-norm to promote sparsity or nuclear norm to promote low-rankness. We assume that the entries of $\A$ are independent and normally distributed and make no assumptions on the noise vector $\z$, other than it being independent of $\A$. Under this generic setup, we derive a general, non-asymptotic and rather tight upper bound on the $\ell_2$-norm of the estimation error $\|\hat\x-\x_0\|_2$. Our bound is geometric in nature and obeys a simple formula; the roles of $\la$, $f$ and $\x_0$ are all captured by a single summary parameter $\delta(\la\paf)$, termed the Gaussian squared distance to the scaled subdifferential.
We connect our result to the  literature and verify its validity through simulations.
\end{abstract}

\input{Introduction}

\input{prelim}

\input{result}

\input{proof}

\section{Future directions}\label{sec:fut}
This paper adds to a recent line of work \cite{StoLASSO,OTH,oymakSimple} which characterizes the $\ell_2$-norm of the estimation error of lasso-type algorithms, when the measurement matrix has i.i.d standard normal entries. This opens many directions for future work, including but not limited to the following: a) analyzing variations of the lasso with the loss function being different than the $\ell_2$-norm of the residual, b) extending the analysis to  measurement ensembles, beyond i.i.d Gaussian.

\bibliography{compbib}

\end{document}

%% file: commands.tex
\newcommand{\beq}{\begin{equation}}
\newcommand{\eeq}{\end{equation}}
\newcommand{\bea}{\begin{align}}
\newcommand{\eea}{{\end{align}}}

\newcommand{\labb}{\lambda_{\text{best}}}

\newcommand{\lam}{\lambda_{{\max}}}

\newcommand{\dlf}{\delta(\la\paf)}
\newcommand{\delf}{\delta(\text{cone}(\paf))}

\newcommand{\Gb}{\mathbf{G}}

\newcommand{\X}{\mathbf{X}}
\newcommand{\A}{\mathbf{A}}

\newcommand{\w}{\mathbf{w}}

\newcommand{\x}{\mathbf{x}}

\newcommand{\g}{\mathbf{g}}
\newcommand{\vb}{\mathbf{v}}

\newcommand{\y}{\mathbf{y}}
\newcommand{\s}{\mathbf{s}}
\newcommand{\z}{\mathbf{z}}
\newcommand{\zo}{\overline{\mathbf{z}}}

\newcommand{\ab}{\mathbf{a}}

\newcommand{\h}{\mathbf{h}}

\newcommand{\Sc}{{\mathcal{S}}}

\newcommand{\Lc}{\mathcal{L}}

\newcommand{\Nn}{\mathcal{N}}
\newcommand{\Cc}{\mathcal{C}}

\newcommand{\Pro}{{\mathbb{P}}}
\newcommand{\E}{{\mathbb{E}}}

\newcommand{\paf}{\pa f(\x_0)}

\newcommand{\la}{{\lambda}}

\newcommand{\pa}{\partial}

\newcommand{\dt}{\text{{dist}}}

\newcommand{\vs}{\vspace}

\newcommand{\nn}{\nonumber}

%% file: Introduction.tex
%\section{Introduction}
\section{Introduction}\label{sec:intro}

\subsection{Motivation}
%\subsection{Estimating structured signal from noisy observations}
We wish to estimate a 
structured
 signal $\x_0\in\mathbb{R}^n$ from a vector of compressed noisy linear observations $\y=\A\x_0+\z\in\mathbb{R}^m$, where $m<n$. What makes the estimation possible in the underdetermined regime, 
is the assumption that $\x_0$ retains a particular structure.
To promote this structure, we associate with it a properly chosen convex function\footnote{See \cite{Cha} for examples and a principled approach to constructing such structure inducing functions.} $f:\mathbb{R}^n\rightarrow\mathbb{R}$. As a motivating example, when $\x_0$ is a sparse vector, 
we can choose $f(\x)=\|x\|_1$. 
A typical approach for estimating $\x_0$ is via convex programming. 
If $f(\x_0)$ is known \emph{a-priori}, then the following convex program yields a reasonable estimate:
\begin{align}\label{eq:classo}
\min_\x~\| \y-\A\x\|_2\quad\text{s.t.}~f(\x)\leq f(\x_0).
\end{align}
It solves for an estimate $\hat\x$ that best fits the vector of observations $\y$ while at the same time retains structure similar to that of $\x_0$. Program \eqref{eq:classo}, with $f(\x)=\|x\|_1$, was introduced in \cite{TibLASSO} by Tibshirani  for estimating sparse signals and is known as the ``lasso" in the statistics literature. In practical situations, prior knowledge of $f(\x_0)$ is typically not available, which makes \eqref{eq:classo} impossible to solve. Instead, one can solve \emph{regularized} versions of it, like,
\begin{align}\label{eq:ell2lasso}
\min_\x~\| \y-\A\x\|_2+\frac{\la}{\sqrt{m}} f(\x),
\end{align}
or
\begin{align}\label{eq:ell22lasso}
\min_\x~\frac{1}{2}\| \y-\A\x\|_2^2+\frac{\tau}{\sqrt{m}} f(\x),
\end{align}
for  nonnegative regularizer parameters $\la$ and $\tau$.
Although very similar in nature, \eqref{eq:ell2lasso} and \eqref{eq:ell22lasso} show in general different statistical behaviors \cite{belloni,OTH}. 
 Lagrange duality ensures that there exist $\la$ and $\tau$ such that 
% the regularized optimization problems \eqref{eq:ell2lasso} and \eqref{eq:ell22lasso}
 they both
  become equivalent to the constrained optimization \eqref{eq:classo}. However, in practice, the challenge lies in tuning the regularizer parameters to achieve good estimation, with as low as possible prior knowledge.
%  without any prior knowledge of $f(\x_0)$. 
Assuming that the entries of the noise vector $\z$ are i.i.d, it is well-known that a sensible choice of $\tau$ in \eqref{eq:ell22lasso} must scale with the standard deviation $\sigma$ of the noise components \cite{candes2006stable,bickel,negahban2012unified}. On the other hand, \eqref{eq:ell2lasso}  eliminates the need to know or to pre-estimate $\sigma$ \cite{belloni}. This fact was first proven by Belloni et. al. in \cite{belloni} 
\footnote{Belloni et. al \cite{belloni} refer to \eqref{eq:ell2lasso} as the ``square-root lasso" to distinguish from the ``standard lasso" estimator \eqref{eq:ell22lasso}. The authors in \cite{OTH} refer to the two estimators in \eqref{eq:ell2lasso} and \eqref{eq:ell22lasso} as the $\ell_2$-lasso and $\ell_2^2$-lasso, respectively.
 }
  for the $\ell_1$-case $f(\cdot)=\|\cdot\|_1$, and has, since then, spurred significant research interest on the analysis of \eqref{eq:ell2lasso}. 

\vs{-0.2cm}
\subsection{Contribution}
In this work, we derive a simple non-asymptotic upper bound on the normalized estimation error $\frac{\|\hat\x-\x_0\|_2}{\|\z\|_2}$ of the regularized estimator \eqref{eq:ell2lasso}, which holds for arbitrary convex regularizers $f(\cdot)$. We assume that the measurement matrix  $\A$ has independent zero-mean normal entries of variance $\frac{1}{m}$.
%, but make no further assumptions on the distribution of the noise vector $\z$. 
For the noise vector $\z$, we only require it being chosen independently of $\A$.

Our upper bound is a simple function of the number of measurements $m$
%, the value of the regularizer parameter $\la$, 
and, of a summary parameter $\delta(\la\paf)$, termed the Gaussian squared distance;  $\delta(\la\paf)$ captures  the structure induced  by $f(\cdot)$, the particular $\x_0$ we are trying to recover and the value of the regularizer parameter $\la$.
For example, when we are interested in a sparse signal, the structure is captured by $f(\cdot)$, whereas the actual sparsity level (i.e. how many entries are zero) is captured by $\x_0$. (Thus $\dlf$ is the same for all $k$-sparse $\x_0$).
 In recent works \cite{Cha,Foygel,TroppEdge,OTH}, $\dlf$ has been calculated for a number of practical regularizers $f(\cdot)$; making use of these results, translates our bound to explicit formulae.
Finally, the constants involved in our result are small and nearly accurate. 
As a byproduct, our bound  provides a guideline on the important practical problem of optimally tuning the regularizer parameter $\la$. 

\subsection{Related work}
There is a significant body of research devoted to the  performance analysis of lasso-type estimators, especially in the context of sparse recovery. A complete review of this literature is beyond the scope and the space of the current paper.   Our result complements and extends recent work in \cite{montanariLasso,StoLASSO,OTH,oymakSimple}. 
The authors in \cite{montanariLasso} obtain an explicit expression for the normalized error of \eqref{eq:ell22lasso}, in an asymptotic setting.  
\cite{StoLASSO} proposes a framework for the analysis of lasso-type algorithms and relies on it to perform a precise analysis of the (worst-case) estimation error of \eqref{eq:classo} for $f(\cdot)=\|\cdot\|_1$. A generalization of this analysis to arbitrary convex regularizers and an extension to the regularized problem \eqref{eq:ell2lasso} was provided in our work\cite{OTH}. In contrast to the present paper, the bounds in \cite{StoLASSO,OTH} require stronger assumptions, namely, an i.i.d. Gaussian noise vector $\z$ and an asymptotic setting where $m$ is large enough. Our latest work \cite{oymakSimple} successfully relaxes both those assumptions and derives sharp bounds on the estimation error of the constrained lasso \eqref{eq:classo}. The result presented in the current paper is a natural extension of that, to the more challenging, and practically important, $\ell_2$-regularized lasso problem \eqref{eq:ell2lasso}. In the context of sparse estimation, our bound recovers the order of best known result in the literature \cite{belloni}, and improves on the constants that appear in it. See Section \ref{sec:comp} for a more elaborate discussion on the relation to these works.

    \setcounter{tocdepth}{2}

%% file: prelim.tex
\section{preliminaries}\label{sec:pre}

%Let $\Bc^{n-1}$ denote the unit $\ell_2$-ball in $\R^n$. 
For the rest of the paper, let $\Nn(\mu,\sigma^2)$ denote the normal distribution of mean $\mu$ and variance $\sigma^2$. Also, to simplify notation, let us write $\|\cdot\|$ instead of $\|\cdot\|_2$.

%For a vector $\g\in\R^m$ with independent $\Nn(0,1)$ entries, we define $\gamma_m:=\E[\|\g\|]$. It is well known (\cite{Gor}) that $\gamma_m=\sqrt{2}\frac{\Gamma(\frac{m+1}{2})}{\Gamma(\frac{m}{2})}$ and $\sqrt{m}\geq \gamma_m\geq \frac{m}{\sqrt{m+1}}$. 

\subsection{Convex geometry}
\subsubsection{Subdifferential}
Let $f:\mathbb{R}^n\rightarrow\mathbb{R}$ be a convex function and $\x_0\in\mathbb{R}^n$ be an arbitrary point that is \emph{not} a minimizer of $f(\cdot)$.
%Following the discussion in Section \ref{sec:intro}, $f(\cdot)$ is to be though as a structure inducing function for the signal $\x_0$. 
The subdifferential of $f(\cdot)$ at $\x_0$ is the set of vectors,
$$\paf = \left\{ \s\in\mathbb{R}^n | f(\x_0+\w) \geq f(\x_0) + \s^T\w, \forall \w\in\mathbb{R}^n  \right\}.$$
%and is always a compact and convex set \cite{Roc70}. 
$\paf$ is a nonempty, convex and compact set \cite{Roc70}. It, also, does not contain the origin since we assumed that $\x_0$ is not a minimizer.
For any nonnegative number $\la\geq0$, we denote the scaled (by $\la$) subdifferential set as $\la\paf=\{\la\s|\s\in\paf\}$. Also, for the conic hull of the subdifferential $\paf$, we  write $\text{cone}(\paf) = \{\s | \s\in\la\paf, \text{ for some } \la\geq 0\}$.

\subsubsection{Gaussian squared distance}\label{sec:dlf}

Let $\Cc\subset\mathbb{R}^n$ be an arbitrary nonempty, convex and closed set. Denote the distance of a vector $\vb\in\mathbb{R}^n$ to $\Cc$, as 
$\dt(\vb,\Cc) = \min_{\s\in\Cc}\|\vb-\s\|$.

\begin{defn}[Gaussian distance] Let $\h\in\mathbb{R}^n$ have i.i.d $\Nn(0,1)$ entries. The Gaussian squared distance of a set $\Cc\subset\mathbb{R}^n$ is defined as $\delta(\Cc) := \mathbb{E}_h\left[ \dt^2(\h,\Cc) \right]$.
%\begin{align*}
%{\omega}(\Cc):=\E_{\h}\left[\sup_{\s\in\Cc}~\h^T\s\right],~~
%\text{ and, }~~
%\eta(\Cc) := \mathbb{E}_{\h}\left[ \dt(\h,\Cc) \right],
%\end{align*}
%respectively.
\end{defn}

%Those notions relate to our setup, through the  Gaussian width of the restricted tangent cone $\omega(T_f(\x)\cap \Bc^{n-1})$ and the Gaussian distance of the scaled subdifferential $\eta(\la\paf)$. 
Our main result in Theorem \ref{thm:main}  upper bounds %\footnote{The measurement matrix $A$ is assumed to have independent standard normal entries.}
 the estimation error of the regularized problem \eqref{eq:ell2lasso} for any value $\la\geq0$, in terms of the Gaussian squared distance of the scaled subdifferential $\delta(\la\paf)$. In a closely related work \cite{oymakSimple}, we show that the error of the constrained problem \eqref{eq:classo} admits a similar bound when $\delta(\la\paf)$ is substituted by $\delta(\text{cone}(\paf))$. In that sense, $\delta(\la\paf)$ and $\delta(\text{cone}(\paf))$ are the fundamental summary components that appear in the characterization of the performance of the lasso-type optimizations \eqref{eq:classo} and \eqref{eq:ell2lasso}. It is worth appreciating this result as an extension to the role that the same quantities play in the noiseless compressed sensing and the proximal de-noising problems. The former, recovers $\x_0$ from noiseless compressed observations $\A\x_0$, by solving $\min\|\x\|_1~\text{s.t. }\A\x=\A\x_0$. \cite{Sto,Cha} showed that $\delf$ number of measurements are sufficient to guarantee successful recovery. More recently, \cite{TroppEdge} proved that these many measurements are also necessary. In proximal de-noising, an estimate of $\x_0$ from noisy measurements $\y=\x_0+\z$, is obtained by solving $\min_{\x}\frac{1}{2}\|\y-\x\|^2+\la\sigma f(\x)$. Under the assumption of the entries of $\z$ being i.i.d $\Nn(0,\sigma^2)$, \cite{oymakProx} shows that $\frac{\mathbb{E}\|\x-\x_0\|^2}{\sigma^2}$ admits a sharp upper bound (attained when $\sigma\rightarrow 0$) equal to $\delta(\la\paf)$. %We elaborate more on the role of those parameters and their role in the estimation error in Section \ref{sec:examples}. 
%But before, it is important to state a result, proven in \cite{Foygel} which relates the 

\subsection{Gordon's Comparison Lemma and Concentration results }
%The notions of the convex geometry summarized in the section above are critical for the presentation and the interpretation of our result. 
In the current section, we outline the main tools that underly the proof of our result.
% Further details on their specific application can be found in Section \ref{sec:proof}.
We begin with a very useful lemma proved by Gordon \cite{Gor} which allows a probabilistic comparison between two Gaussian processes. Here, we use a slightly modified version of the original lemma (see Lemma 5.1 in
%, as appears, along with its proof, in 
 \cite{OTH}). 

\begin{lem}[Comparison Lemma,\cite{Gor}] \label{lem:Gor}
Let $\Gb\in\mathbb{R}^{m\times n}$, $\g\in\mathbb{R}^m$ and $\h\in\mathbb{R}^n$ have i.i.d $\Nn(0,1)$ entries and be independent of each other. Also, let $\mathcal{S}\subset\mathbb{R}^n$ an arbitrary set and $\psi:\Sc\times\mathbb{R}^m\rightarrow\mathbb{R}$ an arbitrary function. For any real $c$,
\begin{align*}
&\Pro\left( \min_{\x\in\Sc} \max_{\|\ab\|=1} \x^T\Gb\ab +\psi(\x,\ab) \geq c \right) \geq \\
&~~\quad2\cdot\Pro\left( \min_{\x\in\Sc} \max_{\|\ab\|=1} \|\x\|\g^T\ab - \h^T\x +\psi(\x,\ab)  \geq c \right) - 1.
\end{align*}
\end{lem}

We further require a standard but powerful result \cite{ledoux} on the concentration of Lipschitz functions of Gaussian vectors. Recall that a function $\psi(\cdot):\mathbb{R}^n\rightarrow\mathbb{R}$ is L-Lipschitz, if for all $\x,\y\in\mathbb{R}^n$, $|\psi(\x)-\psi(\y)|\leq L\|\x-\y\|$.
\begin{lem}[Lipschitz concentration, \cite{ledoux}]\label{lem:Lip}
Let $\h\in\mathbb{R}^n$ have i.i.d $\Nn(0,1)$ entries and $\psi:\mathbb{R}^n\rightarrow\mathbb{R}$ be an L-Lipschitz function. 
%Then, $\text{Var}[\psi(\h)]\leq L^2$.
Then, for all $t>0$, the events $\{\psi(\h) - \E[\psi(\h)]\geq t\}$ and $\{\psi(\h) - \E[\psi(\h)]\leq -t\}$ hold with probability no greater than $\exp(-t^2/(2L^2))$, each. 
%\begin{align*}
%\Pro( \psi(\h) - \E[\psi(\h)]\geq t)\leq \exp(-t^2/(2L^2)),\\
%\Pro( \psi(\h) - \E[\psi(\h)]\leq -t)\leq \exp(-t^2/(2L^2)).
%\end{align*}
\end{lem}

%% file: result.tex
\section{Result}\label{sec:result}
%We are now ready to state our main result.
\begin{thm}\label{thm:main}
Assume $m\geq 2$, $\z\in\mathbb{R}^m$ and $\x_0\in\mathbb{R}^n$ are arbitrary, and, $\A\in\mathbb{R}^{m\times n}$ has i.i.d $\Nn(0,\frac{1}{m})$ entries. Fix the regularizer parameter in \eqref{eq:ell2lasso} to be $\la\geq 0$ and let $\hat\x$ be a minimizer of \eqref{eq:ell2lasso}. Then, for any $0< t\leq (\sqrt{m-1}-\sqrt{\dlf})$, with probability 
$1-{5}\exp(-{t^2}/{32})$, we have,
\begin{align}\label{eq:main}
\| \hat\x - \x_0 \| \leq 2\|\z\|\frac{\sqrt{\dlf}+t}{\sqrt{m-1}-
\sqrt{\dlf}-t}.
\end{align}
\end{thm}

%\emph{Remark}: Our proof technique yields the even tighter upper bound $2\|\z\|\frac{\eta(\la\cdot\paf)+t}{\sqrt{m-1}^2-\eta(\la\cdot\paf)^2-t\sqrt{m-1}}$
\vspace{5pt}
\subsection{Interpretation}\label{sec:int}
Theorem \ref{thm:main} provides a simple, general, non-asymptotic and (rather) sharp upper bound on the 
%estimation
error of the regularized lasso estimator \eqref{eq:ell2lasso}, which also takes into account the specific choice of the regularizer parameter $\la\geq0$. 
In principle, the bound applies to any signal class that exhibits some sort of low-dimensionality (see \cite{oymakSimple} and references therein). 
%The bound holds for an arbitrary convex regularizer $f(\cdot)$ and the role of $f(\cdot)$ is summarized in the Gaussian distance term $\sqrt{\dlf}$.
%, thus, does not necessarily require large number of measurements. 
 It is non-asymptotic and is applicable in any regime of $m$, $\lambda$ and 
%  the summary parameter
   $\dlf$. Also, the constants involved in it are small %as discussed more in detail in Section \ref{sec:compare} and illustrated in the simulation results, 
making it rather tight\footnote{We suspect and is also supported by our simulations (e.g. Figure \ref{fig:plot}) that the factor of 2 in \eqref{eq:main} is an artifact of our proof technique and not essential.}.
%\chris{Talk about simulations figure} 
% (see also Figure \ref{fig:plot}).
%  As a final remark, 

% It is discussed in Section \ref{sec:optimal} how Theorem \ref{thm:main} provides a guideline on the important practical issue of tuning the regularizer parameter $\lambda$.

%\vspace{-25pt}
The Gaussian distance term $\dlf$ summarizes the geometry of the problem and is key in \eqref{eq:main}.
%This is an indication that the estimation is not stable in that regime. 
In \cite{TroppEdge} (Proposition 4.4), it is proven that $\delta(\la\paf)$, when viewed as a function of $\la\geq 0$, is strictly convex, differentiable for $\la>0$
 and achieves its minimum at a unique point. Figure \ref{fig:dlf} illustrates this behavior;
 $\sqrt{m-1}-\sqrt{\dlf}$ achieves its unique  maximum value at some $\la=\labb$, it is strictly increasing for $\la<\labb$ and strictly decreasing for $\la>\labb$. 
 For the bound in \eqref{eq:main} to be at all meaningful, we require $m>\min_{\la\geq0}\dlf=\delta(\labb\paf)$. This is perfectly in line with our discussion in Section \ref{sec:dlf}, and translates to the number of measurements being large enough to at least guarantee noiseless recovery \cite{Cha,TroppEdge,Foygel,oymakProx,maleki}. 
Lemma 8.1 in \cite{OTH} proves that there exists a unique $\la_{max}$ satisfying $\la_{max}>\labb$ and $\sqrt{\delta(\la_{max}\paf)}=\sqrt{m-1}$. Similarly, when $m\leq n$, there exists unique $\la_{min}<\labb$ satisfying $\sqrt{\delta(\la_{min}\paf)}=\sqrt{m-1}$. From this, it follows that $\sqrt{m-1}>\sqrt{\dlf}$ if and only if $\la\in(\la_{min},\la_{max})$. This is exactly the range of values of the regularizer parameter $\la$ for which  \eqref{eq:main} is meaningful; see also Figure \ref{fig:dlf}.

 The region $(\la_{min},\la_{max})$, which our bound characterizes, contains $\labb$, for which, the bound in \eqref{eq:main} achieves its minimum value since it is strictly increasing in $\dlf$. Note that deriving $\labb$ does not require knowledge of any properties (e.g. variance) of the noise vector. All it requires is knowledge of the particular structure of the unknown signal. For example, in the $\ell_1$-case, $\labb$ depends only on the sparsity of $\x_0$, not $\x_0$ itself, and in the nuclear norm case, it only depends on the rank of $\x_0$,  not $\x_0$ itself.

\begin{figure}
  \begin{center}
%  \hspace{-20pt}
%{\includegraphics[scale=0.4]{dlf.eps}}
{\includegraphics[width=9cm, height=5cm]{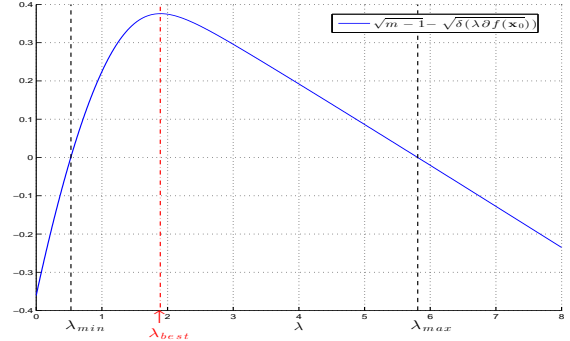}}
  \end{center}
  \vspace{-10pt}
  \caption{\footnotesize{Illustration of the denominator $\sqrt{m-1}-\sqrt{\dlf}$ in \eqref{eq:main} as a function of $\la\geq 0$. The bound is meaningful for $\la\in(\la_{min},\la_{max})$ and attains its minimum value at $\labb$. The $y$-axis is normalized by $\sqrt{n}$.
  }}
%  \vspace*{-10pt}
\label{fig:dlf}
\end{figure}

%, and becomes practical as long as the Gaussian term $\sqrt{\dlf}$ 
\subsection{Application to sparse and low-rank estimation}
Any bound on $\delta(\la\paf)$ 
%%The authors in \cite{Cha,TroppEdge,Foygel,OTH} have computed explicit expressions or upper bounds for $\delta(\la\paf)$, for a number of different functions $f(\cdot)$.
%Such explicit expression or upper bounds of $\delta(\la\paf)$ 
%Such upper bounds,  for specific examples of structures of $\x_0$, have been explicitly derived or can be deduced with minor effort from recent works (e.g. \cite{Cha,TroppEdge,OTH,Foygel}). 
%Those results, 
translates, through Theorem \ref{thm:main}, into an upper bound on the estimation error of \eqref{eq:ell2lasso}. Such bounds have been recently derived in \cite{Cha,TroppEdge,Foygel,OTH}, for a variety of structure-inducing functions $f(\cdot)$.
%In principle, our result would apply to any signal class that exhibits some sort of low-dimensionality ( \cite{oymakSimple} and references therein). 
For purposes of illustration and completeness, we review here those results for the celebrated cases of sparse and low-rank estimation.
% estimation of sparse and bock-sparse signals, and of low-rank matrices.

\subsubsection{Sparse signals}\label{sec:sparse} Suppose $\x_0$ is a $k$-sparse signal and  $f(\cdot)=\|\cdot\|_1$. Denote by $S$ the support set of $\x_0$,
% such that $|S|=k$, 
and by $S^c$ its complement. The subdifferential at $\x_0$ is\cite{Roc70},
 $$
 \paf = \{\s\in\mathbb{R}^n | \|\s\|_\infty\leq 1\text{ and } \s_i=\text{sign}((\x_0)_i),\forall i\in S \}.
 $$
 Let $h\in\mathbb{R}^n$ have i.i.d $\Nn(0,1)$ entries and define
% note the $\ell_1$-shrinkage function as 
 $$
 \text{shrink}(\chi,\la) = \begin{cases} \chi - \la & ,\chi>\la, \\  0 & ,-\la\leq\chi\leq\la, \\ \chi+\la   & ,\chi<-\la.\end{cases}
 $$
  Then, $\delta(\la\partial\|\x_0\|_1)$ is equal to (\cite{Cha,TroppEdge})
 \begin{align}
&\sum_{i\in S}\mathbb{E}[( \h_i-\la\text{sign}((\x_0)_i) )^2]+\sum_{i\in S^c}\mathbb{E}[\text{shrink}^2(\h_i,\la) ]=\notag\\
&k(1+\la^2) + (n-k)[(1+\la^2)\text{erfc}(\frac{\la}{\sqrt{2}})-\sqrt{\frac{2}{\pi}}\la\exp(-\frac{\la^2}{2})],\label{eq:closed}
 \end{align}
 where $\text{erfc}(\cdot)$ denotes the standard complementary error function. 
 Note that $\delta(\la\partial\|\x_0\|_1)$ depends only on $n,\la$ and $k=|S|$, and \emph{not} explicitly on S itself (which is not known).
 Substituting the expression in \eqref{eq:closed} in place of the $\dlf$ term in \eqref{eq:main}, yields an explicit expression for our upper bound, in terms of $n$, $m$, $k$ and $\la$. 
 A simpler upper bound which does not involve error functions is obtained in Table 3 in \cite{OTH}, and is given by  
 \begin{align}\label{eq:gr}
\delta(\la\partial\|\x_0\|_1) \leq (\la^2+3)k, \text{ when } \la\geq\sqrt{2\log(\frac{n}{k})}.
 \end{align} 
% Although explicit, this expression is not in closed-form. For a simple and closed-form bound we use a result from \cite{OTH} (see Appendix H), where it is shown that the quantity in \eqref{eq:closed} is upper bounded by $(\la^2+3)k$, when $\la\geq\sqrt{2\log(\frac{n}{k})}$.
 Analogous expressions and closed-form upper bounds can be obtained when $\x_0$ is block-sparse \cite{stoBlock,OTH}.

\subsubsection{Low-rank matrices} Suppose $\X_0\in\mathbb{R}^{\sqrt{n}\times \sqrt{n}}$ is a rank-$r$ matrix and $f(\cdot)$ is the nuclear norm (sum of singular values). %Although more involved than the corresponding derivations in the case of sparse vector estimation,  an explicit expression for $\delta(\la\paf)$ just like \eqref{eq:closed} can be obtained \cite{OymRank}. In \cite{OTH}, it is further proven that $\delta(\la\paf)$ is upper bounded by
%
%We can give expression analogous to both \eqref{eq:closed} and \eqref{eq:gr}. Due to space limitations we only mention that, analogous to \eqref{eq:gr}, $\dlf$ is upper bounded by $\la^2r+2\sqrt{n}(r+1)$, for $\la\geq 2n^{(1/4)}$ \cite{OTH}.
An upper bound to $\dlf$, analogous to \eqref{eq:gr}, is the following: $\la^2r+2\sqrt{n}(r+1)$, for $\la\geq 2n^{(1/4)}$ \cite{OTH}
.
%\subsubsection{Block-sparse signals}

%\vspace{-4pt}
\subsection{Comparison to related work}\label{sec:comp}
\subsubsection{Sparse estimation} Belloni et al. \cite{belloni} were the first to prove error guarantees for the $\ell_2$-lasso \eqref{eq:ell2lasso}. Their analysis shows that the estimation error is of order $O\left(\sqrt{\frac{k\log(n)}{m}}\right)$, when $m=\Omega(k \log n )$ and $\la>\sqrt{2\log(2n)}$\footnote{\cite{belloni} also imposes a ``growth restriction" on $\la$, which agrees with the fact that our bound becomes vacuous for $\la>\lam$ (see Section \ref{sec:int}).}.
% with the discussion in Section \ref{sec:int} that $\la$ should be no larger than $\lam$ for stable estimation}
% estimation becomes unstable once $\la>\lam$, for some value $\lam$. }. 
Applying $\la=\sqrt{2\log(\frac{n}{k})}$ in \eqref{eq:gr} and Theorem \ref{thm:main} yields the same order-wise error guarantee. Our result is non-asymtpotic and involves explicit coefficients, while the result of \cite{belloni} is applicable to more general constructions of the measurement matrix $\A$.
% other than the i.i.d Gaussian assumption of our work.
%
%
% Following our comments in Section \ref{sec:sparse}, Theorem \ref{thm:main} is fully consistent with this result. At the same time, it involves small and accurate constants, it is non-asymptotic and has a simple expression; the entire geometry of the problem is captured by the gaussian distance to the scaled set of subdifferential.  On the other hand, the result of \cite{belloni} is applicable to more general constructions of the measurement matrix $\A$, other than the i.i.d Gaussian assumption of our work.

\subsubsection{Comparison to the constrained lasso}
%We repeat a result from \cite{oymakSimple} on the estimation performance of the constrained lasso \eqref{eq:classo}. 
Under the same assumptions as in Theorem \ref{thm:main}, it is proven in \cite{oymakSimple} that, for any $0< t\leq \sqrt{m-1}-\sqrt{\delf}$, with probability 
$1-{6}\exp(-{t^2}/{26})$, the estimation error $\| \hat\x - \x_0 \|$ of \eqref{eq:classo} is upper bounded as follows,
\begin{align}%\label{eq:classo_main}
\| \hat\x - \x_0 \|\leq \|\z\|\frac{\sqrt{m}}{\sqrt{m-1}}\frac{\sqrt{\delf}+t}{\sqrt{m-1}-
\sqrt{\delf}-t}.\nn
\end{align}
Comparing this to \eqref{eq:main} reveals the similar nature of the two results. Apart from a factor of $2$ in \eqref{eq:main}, the upper bound on the error of the regularized lasso \eqref{eq:ell2lasso} for fixed $\la$, is essentially the same as the upper bound on the error of the constrained lasso \eqref{eq:classo}, with $\delf$ replaced by $\dlf$. Recent works \cite{TroppEdge,Foygel,oymakProx} prove that $\delf\approx\min_{\la\geq0}\dlf=\delta(\labb\paf)$. Our bound, then, suggests that setting $\la=\labb$ in \eqref{eq:ell2lasso} achieves performance almost as good as that of 
%the constrained lasso 
\eqref{eq:classo}. 

%\vspace{-10pt}

\subsubsection{Sharp error bounds}
\cite{OTH} performs a detailed analysis of the regularized lasso problem \eqref{eq:ell2lasso} under the additional assumption that the entries of the noise vector $\z$ are distributed $\Nn(0,\sigma^2)$. In particular, when $\sigma\rightarrow 0$ and $m$ is large enough, they prove that with high probability,
\begin{align}\label{eq:oth}
\|\hat\x-\x_0\|\approx\|\z\|\frac{\sqrt{\dlf}}{\sqrt{m-\dlf}},
\end{align}
for $\la$ belonging to a particular subset of $(\la_{min},\la_{max})$. As expected, our bound in Theorem \ref{thm:main} is larger than the term in \eqref{eq:oth}. However, apart from a factor of $2$, it only differs from the quantity in \eqref{eq:oth} in the denominator, where instead of $\sqrt{m-\dlf}$, we have the smaller $\sqrt{m-1}-\sqrt{\dlf}$. This difference becomes insignificant and indicates that our bound is  rather tight when $m$ is large.
Although the authors in \cite{OTH} conjecture that \eqref{eq:oth} upper bounds the estimation error for arbitrary values of the noise  variance $\sigma^2$, they do not prove so. In that sense, and to the best of our knowledge, Theorem \ref{thm:main} is the first rigorous upper bound on the estimation error of \eqref{eq:ell2lasso}, which holds for general convex regularizers, is non-asymptotic and requires no assumption on the distribution of $\z$.
\subsection{Simulation results}
Figure \ref{fig:plot} illustrates the bound of Theorem \ref{thm:main}, which is given in red for $n=340$, $m=140$, $k=10$ and for $\A$ having $\Nn(0,\frac{1}{m})$ entries. The upper bound from \cite{OTH}, which is asymptotic in $m$ and only applies to i.i.d Gaussian $\z$, is given in black. In our simulations, we assume $\x_0$ is a random unit norm vector over its support and consider both i.i.d $\Nn(0,\sigma^2)$, as well as, non-Gaussian noise vectors $\z$.
%For the noise vector $\z$ we consider both the scenario in which it has i.i.d $\Nn(0,\sigma^2)$ entries and the scenario in which it is non-Gaussian.
We have plotted the realizations of the normalized error for different values of $\la$ and $\sigma$. As noted, the bound in \cite{OTH} is occasionally violated since it requires very large $m$, as well as, i.i.d Gaussian noise.
%As expected, it is also violated when the noise vector is not i.i.d Gaussian. 
On the other hand, the bound given in \eqref{eq:main} always holds. 
%It will also hold had we considered non-Gaussian $\z$, whereas for such $\z$ there is no guarantee that the bound in \cite{OTH} holds.

%
%% Adding to it and in contrast to \cite{OTH}, Theorem \ref{thm:main} is  non-asymptotic and holds under no assumptions on the distribution of the noise vector. 
%See also Figure \ref{fig:plot} for an illustration of formula \eqref{eq:oth} proposed in \cite{OTH} and how it compares to our bound \eqref{eq:main}.

%\subsection{Tuning the regularizer parameter}\label{sec:optimal}
%Theorem \ref{thm:main} provides a guideline on the important practical issue of optimally tuning the regularizer parameter $\lambda$. It suggests setting $\la$ to $\labb:=\min_{\la\geq 0}\dlf$ (see also Figure \ref{fig:plot}). Observe that this choice of $\la$ does not require knowledge of any properties of the error vector.

\begin{figure}
  \begin{center}
%  \hspace{-20pt}
%{\includegraphics[width=9.3cm]{plott2.eps}}\\
%{\includegraphics[scale=0.43]{smallm.eps}}\\
%{\includegraphics[scale=0.43]{smallm2.eps}}
{\includegraphics[width=9.5cm, height=6cm]{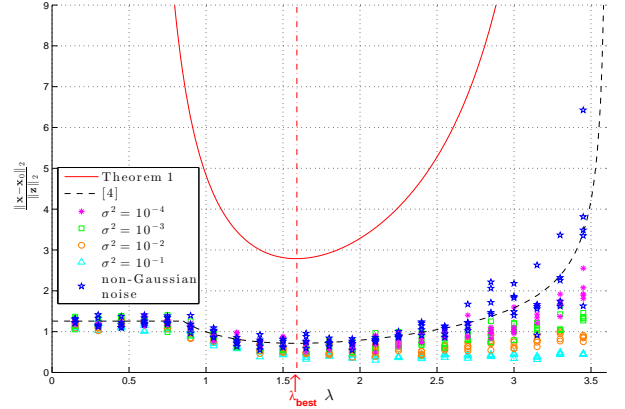}}
  \end{center}
  \vspace{-10pt}
  \caption{\footnotesize{ 
 The normalized error of \eqref{eq:ell2lasso} as a function of $\lambda$.
  }}
%  We consider sparse signal recovery with
% $n=340$, $m=140$ and $k=4$
%%   $n=500$, $m=240$ and $k=5$
%   . Nonzero entries of $\x_0$ are generated i.i.d. $\Nn(0,1)$ and are, then, normalized to ensure unit norm. $\z$ and $\A$ have independent entries distributed $\Nn(0,\sigma^2)$ and $\Nn(0,\frac{1}{m})$, respectively. We fix $\sigma$ and vary $\la$ from 0 to 6. For each choice of $\la$ we solve \eqref{eq:ell2lasso} 5 times and plot the measured normalized estimation error. We repeat for several values of the noise variance $\sigma^2.$ On red, we plot the theoretical upper bound proved in Theorem \ref{thm:main}.  The dashed curve shown in black, follows the formula in \eqref{eq:oth} which predicts the estimation error when $\sigma\rightarrow 0$ and $m$ large enough \cite{OTH}. Note, that \cite{OTH}  does not prove that this is indeed an upper bound on the error for arbitrary $\sigma$. }}
%%  \vspace*{-10pt}
\label{fig:plot}
\end{figure}

%\subsubsection{Least-Squares}
%Theorem \ref{thm:main} admits a nice and insightful interpretation in terms of the ordinary least-squares estimator.
%%We will now argue that, one can easily interpret our results when the system $\y=\A\x_0+\z$ is seen as an $m\times \Delxf^2$ system rather than $m\times n$.
%For this, consider the least-squares problem where one simply estimates  
%$\hat\x = \arg\min_\x\|\y-\A\x\|.$
%Assuming $m>n$, $\hat\x=(\A^T\A)^{-1}\A^T\y$, and it can be shown that
%\beq
%\|\x^*-\x_0\|=(\z^T\A(\A^T\A)^{-2}\A^T\z)^{\frac{1}{2}}\leq \frac{\|\bu(\z,\text{Range}(\A))\|}{\sigma_{min}(\A)}\nn,%,\label{inversion}
%\eeq
%where, $\bu(\z,\text{Range}(\A))$ denotes the projection of $\z$ onto the range space of $\A$ and $\sigma_{min}(\A)$ denotes the minimum singular value of $\A$.
%Assuming $\A$ has i.i.d. $\Nn(0,\frac{1}{m})$ entries, $\sigma_{min}(\A)\approx 1-\sqrt{\frac{n}{m}}$ and $\|\bu(\z,\text{Range}(\A))\|\approx \sqrt{\frac{n}{m}}\|\z\|$. Consequently,
%\beq
%\|\x^*-\x_0\|\lesssim \|\z\|\frac{\sqrt{n}}{\sqrt{m}-\sqrt{n}}\nn.%\label{ls1}.
%\eeq
%Comparing this to our bound in \eqref{eq:main} (again ignoring the factor of 2 and assuming $\sqrt{m-1}\approx\sqrt{m}$), conveys the following simple message: 
%the error formula that characterizes the lasso-type estimator \eqref{eq:ell2lasso} is the same as that of the ordinary least squares after just replacing the ambient dimension $n$ of the space by the gaussian squared distance term $\dlf$.

%% file: proof.tex
\section{Proof of Theorem \ref{thm:main}}\label{sec:proof}

%For the purpose of proving the error bound in Theorem \ref{thm:main}, it is more convenient to assume that the entries of the measurement matrix $\A$ have variance $1$

It is convenient to rewrite \eqref{eq:ell2lasso} in terms of the error vector $\w=\x-\x_0$ as follows:
\begin{align}\label{eq:ell2}
\min_\w \|\A\w-\z\| + \frac{\la}{\sqrt{m}}( f(\x_0+\w) - f(\x_0)).
\end{align} 
Denote the solution of \eqref{eq:ell2} by $\hat\w$. Then, $\hat\w=\hat\x-\x_0$ and \eqref{eq:main} bounds $\|\hat\w\|$. To simplify notation, for the rest of the proof, we denote the value of that upper bound  as 
\begin{align}\label{eq:ell}
\ell(t):=2\|\z\|\frac{\sqrt{\dlf}+t}{\sqrt{m-1}-\sqrt{\dlf}-t}.
\end{align}
\vs{-0.cm}
It is easy to see that the optimal value of the minimization in \eqref{eq:ell2} is no greater than $\|\z\|$. Observe that $\w=\mathbf{0}$ achieves this value. However, Lemma \ref{lem:main} below shows that if we constrain the minimization in \eqref{eq:ell2} to be only over vectors $\w$ whose norm is greater than $\ell(t)$, then the resulting optimal value is (with high probability on the measurement matrix $\A$) strictly greater than $\|\z\|$. Combining those facts yields the desired result, namely $\|\hat\w\|\leq\ell(t)$. Thus,
%the desired upper bound on the estimation error as stated in \eqref{eq:main},
 it suffices to prove Lemma \ref{lem:main}.

\begin{lem}\label{lem:main}
Fix some $\la\geq 0$ and $0< t\leq (\sqrt{m-1}-\sqrt{\dlf})$. Let $\ell(t)$ be defined as in \eqref{eq:ell}. Then, with probability 
$1-5\exp(-{t^2}/{32})$, we have,
%\vspace{-0.2cm}
\begin{align}\label{eq:ell2cond}
\min_{\|\w\|\geq\ell(t)}\{ \|\A\w-\z\| + \frac{\la}{\sqrt{m}}( f(\x_0+\w) - f(\x_0)) \} > \|\z\|.
\end{align} 
\end{lem}
%Following the discussion above, it suffices to prove Lemma \ref{lem:main}. This is the subject of the next section.
\vs{-0.2cm}
\subsection{Proof of Lemma \ref{lem:main}}
Fix $\la$ and $t$, as in the statement of the lemma. From the convexity of $f(\cdot)$,
$f(\x_0+\w) - f(\x_0) \geq \max_{\s\in\paf} \s^T\w$.
%\begin{align}
%f(\x_0+\w) - f(\x_0) \geq \max_{\s\in\paf} \s^T\w.\nn
%\end{align}
Hence, it suffices to prove that w.h.p. over $\A$,
\begin{align*}%\label{eq:ell2condapprox}
\min_{\|\w\|\geq\ell(t)}\{ \sqrt{m}\|\A\w-\z\| + \max_{\s\in\la\paf} \s^T\w \} > \sqrt{m}\|\z\|.
\end{align*} 
%The main tool for proving \eqref{eq:ell2condapprox} is Gordon's lemma, which allows us to compare two Gaussian processes. That lemma appears  in its original format in \cite{Gor} and has proven to be a very powerful tool in the compressed sensing literature (e.g. \cite{Sto1,Oym,Cha,Foygel,StoLASSO,OTH}). Here, we make use of a slightly modified version of the original lemma, which can be found, along with its proof, in \cite{OTH} (cf. Lemma 5.1).
%
%\begin{lem}[Comparison Lemma,\cite{Gor},\cite{OTH}] \label{lem:Gor}
%Let $\Gb\in\mathbb{R}^{m\times n}$, $\g\in\mathbb{R}^m$ and $\h\in\mathbb{R}^n$ have independent standard normal entries and be independent of each other. Also, let $\mathcal{S}\subset\mathbb{R}^n$ be an arbitrary set and $\psi:\Sc\times\mathbb{R}^m\rightarrow\mathbb{R}$ be an arbitrary function. Then, for any $c\in\mathbb{R}^n$, 
%\begin{align*}
%&\Pro\left( \min_{\x\in\Sc} \max_{\|\ab\|=1} \x^T\Gb\ab -\psi(\x,\ab) \geq c \right) \geq \\
%&\qquad\quad\Pro\left( \min_{\x\in\Sc} \max_{\|\ab\|=1} \|\x\|\g^T\ab - \h^T\x -\psi(\x,\ab)  \geq c \right).
%\end{align*}
%\end{lem}
We begin with applying Gordon's Lemma \ref{lem:Gor} to the optimization problem in the expression above. Define, $\zo=\sqrt{m}\z$, rewrite $\|\A\w-\z\|$ as $\max_{\|\ab\|=1}\{\ab^T\A\w - \ab^T\z\}$ and, then, apply Lemma \ref{lem:Gor} with $\Gb=\sqrt{m}\A$, $\Sc = \{\w~|~\|\w\|\geq\ell(t)\}$ and $\psi(\w,\ab) = -\ab^T\zo + \max_{\s\in\la\paf} \s^T\w$. This leads to the following statement:
\begin{align*}
&\Pro\left(~ \eqref{eq:ell2cond} \text{ is true}~\right) \geq 
2\cdot\Pro\left(~ \Lc(t;\g,\h) > \|\zo\|~ \right)-1,
\end{align*}
where, $\Lc(t;\g,\h)$ is defined as 
\begin{align}\label{eq:simple}
\min_{\|\w\|\geq \ell(t)} \max_{\|\ab\|=1}\{ (\|\w\|\g-\zo)^T\ab - \min_{\s\in\la\paf} (\h-\s)^T\w \}.
\end{align} 
% The corresponding optimization problem in the right-hand side of the inequality in Lemma \ref{lem:Gor} becomes
%\begin{align}\label{eq:simple}
%&\Lc(t;\g,\h) := \notag
%\\
%&~\min_{\|\w\|\geq \ell(t)} \max_{\|\ab\|=1}\left\{ (\|\w\|\g-\zo)^T\ab - \min_{\s\in\la\paf} (\h-\s)^T\w \right\},
%\end{align} 
%and the following is, then, true:
%\begin{align*}
%&\Pro\left(~ \eqref{eq:ell2cond} \text{ is true}~\right) \geq 
%2\cdot\Pro\left(~ \Lc(t;\g,\h) > \|\zo\|~ \right)-1. 
%\end{align*}
In the remaining, we analyze the simpler optimization problem defined in \eqref{eq:simple}, and  prove that $\Lc(t;\g,\h) > \|\zo\|$ holds with probability $1-\frac{5}{2}\exp(-t^2/32)$. We begin with simplifying the expression for $\Lc(t;\g,\h)$, as follows:
\begin{align}
&\Lc(t;\g,\h) = \min_{\|\w\|\geq \ell(t)}\{ \|\|\w\|\g-\zo\| - \min_{\s\in\la\paf} (\h-\s)^T\w \}\nn%\label{eq:r1}
\\
&~~~~~~~~~~~= \min_{\alpha \geq \ell(t)}\left\{ \|\alpha\g-\zo\| - \alpha\dt(\h,\la\paf) \right\}\nn%\label{eq:r2}
\\
&= \
\min_{\alpha\geq\ell(t)}\{\sqrt{\alpha^2\|\g\|^2 + \|\zo\|^2 - 2\alpha\g^T{\zo}} %\\
%&~~~~~~~~\qquad\qquad\qquad
- \alpha\dt(\h,\la\paf)\}.
 \label{eq:p}
\end{align}
%\eqref{eq:r1}
The first equality above
 follows after performing the trivial maximization over $\ab$ in \eqref {eq:simple}. %Reducing  \eqref{eq:r1} to \eqref{eq:r2} 
 The second, uses the fact that $\max_{\|w\|=\alpha}\min_{\s\in\la\paf}(\h-\s)^T\w =  \min_{\s\in\la\paf}\max_{\|w\|=\alpha}(\h-\s)^T\w= \alpha\cdot \dt(\h,\la\paf)$, for all $\alpha\geq 0$. For a proof of this  see Lemma E.1 in \cite{OTH}.
 
%  Therein, it is shown that $\max_{\|w\|=\alpha}\min_{\s\in\la\paf}(\h-\s)^T\w = \min_{\s\in\la\paf}\max_{\|w\|=\alpha}(\h-\s)^T\w$, and the later is not hard to see to be equal to $\alpha\cdot \dt(\h,\la\paf)$.
%From \eqref{eq:r2}, since $\ell(t)>0$, we can proceed one step further and lower bound $\Lc(t;\g,\h)$ with $\hat\Lc(t;\g,\h)$ defined as follows:

%%From \eqref{eq:r2}, since $\ell(t)>0$, we can proceed by rewriting $\Lc(t;\g,\h)$ as follows:
%%\begin{align}
%%&\Lc(t;\g,\h) = \notag
%%\min_{\alpha\geq\ell(t)}\Big\{\sqrt{\alpha^2\|\g\|^2 + \|\zo\|^2 - 2\alpha\g^T{\zo}} \\
%%&~~~~~~~~\qquad\qquad\qquad- \alpha\cdot\dt(\h,\la\paf)\Big\} \label{eq:p},
%%\end{align}
%where for any real number $\chi\in\mathbb{R}$ we denote $(\chi)_+ = \max\{\chi,0\}$.
Next, we show that $\Lc(t;\g,\h)$ is strictly greater than $\|\zo\|$ with the desired high probability over realizations of $\g$ and $\h$. 
Consider the event $\mathcal{E}_t$ of $\g$ and $\h$ satisfying all three conditions listed below,
\begin{subequations}\label{eq:cond}
\begin{align}
&1.~\|\g\| \geq \gamma_m - t/4,\label{eq:cond1}\\
&2.~\dt(\h,\la\paf)\leq\sqrt{\dlf} + t/4,\\
&3.~ \g^T\zo\leq (t/4)\|\zo\|.
\end{align}
\end{subequations}
In \eqref{eq:cond1} we have denoted $\gamma_m:=\E[\|\g\|]$; it is well known that $\gamma_m=\sqrt{2}\frac{\Gamma(\frac{m+1}{2})}{\Gamma(\frac{m}{2})}$ and $\gamma_m\leq\sqrt{m}$. 
The conditions in \eqref{eq:cond} hold with high probability. In particular, the first two hold with probability no less than $1-\exp(-t^2/32)$. This is because the $\ell_2$-norm and the distance function to a convex set are both 1-Lipschitz functions and, thus, Lemma \ref{lem:Lip} applies. The third condition holds with probability at least $1-(1/2)\exp(-t^2/32)$, since $\g^T\zo$ is statistically identical to $\Nn(0,\|\zo\|^2)$. Union bounding yields,
\begin{align}\label{eq:pro}
\Pro(\mathcal{E}_t)\geq1-(5/2)\exp(-t^2/32).
\end{align}
Furthermore, Lemma \ref{lem:det}, below, shows that if $\g$ and $\h$ are such that $\mathcal{E}_t$ is satisfied, then $\Lc(t;\g,\h)>\|\zo\|$. This, when combined with \eqref{eq:pro} shows that $\Pro(\Lc(t;\g,\h)>\|\zo\|)\geq1-(5/2)\exp(-t^2/32)$, completing the proof of Lemma \ref{lem:main}.
% We state the result as a lemma below, and proceed with proving it.
%
%Having derived a simplified expression for $\Lc(t;\g,\h)$, we now show that $\Lc(t;\g,\h)>\|\z\|$ with the desired high probability over realizations of $\g$ and $\h$. We state the result as a lemma below, and proceed with proving it.
\begin{lem}\label{lem:det}
Fix any $0< t\leq (\sqrt{m-1}-\sqrt{\dlf})$. Suppose $\g$ and $\h$ are such that \eqref{eq:cond} holds
 and recall the definition of $\Lc(t;\g,\h)$ in \eqref{eq:p}. 
 Then,
$\Lc(t;\g,\h) > \|\zo\|.$
\end{lem}
\begin{proof}
%Fix any $t$ as in the statement of the lemma. For all $\alpha\geq\ell(t)> 0$,
%\begin{align}
%\|\alpha\g-\z\| &= \sqrt{\alpha^2\|\g\|^2 + \|\z\|^2 - 2\alpha\g^T{\z}}\notag
% \\
%&\geq \sqrt{\alpha^2\|\g\|^2 + \|\z\|^2 - 2\alpha(\g^T{\z})_+}\label{eq:p},
%\end{align}
%where, for any real number $\chi\in\mathbb{R}$ we denote $(\chi)_+ = \max\{\chi,0\}$.
%Next, we make use the following concentration results, which hold for any $\eps\geq 0$:
%\begin{align}
%&\bullet~\Pro\left(\|\g\| \geq \gamma_m - \eps\right)\geq 1-\exp{(-\eps^2/2)},\notag\\
%&\bullet~\Pro\left( \dt(\h,\la\paf)\leq\sqrt{\dlf} + \eps\right) \geq 1-\exp(-\eps^2/2),\notag\\
%&\bullet~\Pro\left( \g^T\z\leq \eps\|\z\| \right) \geq 1-(1/2)\exp(-\eps^2/2)\notag.
%\end{align}
%The first two hold because the $\ell_2$-norm and the distance function to a convex set are both 1-Lipschitz functions. The third bound follows from the fact that $\g^T\z$ is statistically identical to $\Nn(0,\|\z\|^2)$. 
Take any $\alpha\geq\ell(t)>0$. Following from \eqref{eq:cond}, we have that the objective function of the optimization in \eqref{eq:p} is lower bounded by
\begin{align}
&\phi(\alpha)=\notag\\
&{\sqrt{\alpha^2(\gamma_m-\frac{t}{4})^2 + \|\zo\|^2 - \frac{1}{2}\alpha\|\zo\|t}-\alpha(\sqrt{\dlf} + \frac{t}{4})}.\nn%\label{eq:rh}
\end{align}
We will show that $\phi(a)>\|\zo\|$, for all $\alpha\geq\ell(t)$, and this will complete the proof. Starting with the desired condition $\phi(\alpha)>\|\zo\|$, using the fact that $\alpha>0$ and performing some algebra, we have the following equivalences,
\begin{align}\notag
&\phi(a)>\|\zo\| \Leftrightarrow 
\alpha^2(\gamma_m-t/4)^2 + \|\zo\|^2 - (1/2)\alpha\|\zo\|t> \\
&\qquad\qquad\qquad\qquad\qquad~~~(\alpha(\sqrt{\dlf} + t/4) + \|\zo\|)^2\notag
\\&\Leftrightarrow \alpha>\frac{2\|\zo\|(\sqrt{\dlf}+t/2)}{\gamma_m^2-\dlf-\frac{t}{2}(\gamma_m+\sqrt{\dlf})}.\label{eq:safi}
\end{align}
Observing that $\gamma_m^2> \sqrt{m}\sqrt{m-1}$ \cite{gamma}, $\gamma_m\leq\sqrt{m}$ and $\sqrt{\dlf}<\sqrt{m}$,
%Now, the following three are true:
%\begin{align*}
%&\bullet~\gamma_m^2>\gamma_{m-1}\gamma_{m+1}\Rightarrow\gamma_m^2> \sqrt{m}\sqrt{m-1}, \text{\cite{gamma}}, \\
%&\bullet~\dlf<\sqrt{m-1}\sqrt{\dlf}<\sqrt{m}\sqrt{\dlf}, \\
%&\bullet~\gamma_m+\sqrt{\dlf}<\sqrt{m}+\sqrt{m-1}<2\sqrt{m}.
%\end{align*}
%From these,
 it can be shown that $\ell(t)$ is strictly greater than the expression in the right hand side of \eqref{eq:safi}. Thus, for all $\alpha\geq\ell(t)$, we have $\phi(\alpha)>\|\zo\|$, as desired.

%
%Now, since $\gamma_m>\sqrt{\dlf}$, we have $\sqrt{\dlf}^2+\frac{t}{2}(\gamma_m+\sqrt{\dlf} < \gamma_m(\sqrt{\dlf} + t)$. Following this, it is not hard to see that $\ell(t)$ is strictly greater than the expression in the right hand side of \eqref{eq:safi}. Thus, for all $\alpha\geq\ell(t)$, we have $\phi(\alpha)>\|\zo\|$, as desired.

%Analyzing $\phi(\alpha)$ will lead us to the desired result. 
%
%First, we show that $\phi(\alpha)>0$, for all $\alpha\geq\ell(t)$. Take any $\alpha\geq\ell(t)$. Using the fact that $\gamma_m>\sqrt{\dlf}$ we find 
%\begin{align}
%\alpha&\geq\|\z\|\frac{t}{\gamma_m^2-\gamma_m\sqrt{\dlf}-2t\gamma_m}\notag\\
%&>\|\z\|\frac{t}{\gamma_m^2-\sqrt{\dlf}^2-t(\gamma_m+\sqrt{\dlf})}.\label{eq:cond}
%\end{align}
%A bit of algebra shows that \eqref{eq:cond} is equivalent to $\sqrt{\alpha^2(\gamma_m-t/2)^2-\alpha\|\z\|t}>\alpha(\sqrt{\dlf}+t/2)$. Thus, $\phi(\alpha)\geq 0$, as desired.

\end{proof}